\documentclass{endm}
\usepackage{endmmacro}
\usepackage{graphicx}
\usepackage{amsmath}

\newcommand{\Hom}{\mathop{\to}}
\newcommand{\FullHom}{\mathop{\xrightarrow{_F}}}
\newcommand{\FullLeq}{\mathop{\leq^F}}

\newcommand{\Graphs}[0]{\ensuremath{\mathsf{Graphs}}}

\def\Rel{\mathop{\mathrm{Rel}}\nolimits}

\let\Homeq\sim
\newcommand{\FullHomeq}{\mathop{\Homeq^{F}}}

\def\str#1{\mathbf {#1}}

\def\arity#1{a(\rel{}{#1})}

\def\rel#1#2{R_{\mathbf{#1}}^{#2}}

\begin{document}

\pagestyle{plain}
\begin{frontmatter}

\title{Gaps in full homomorphism order}

\author{Ji\v r\' \i{} Fiala\thanksref{g1}\thanksref{fiala@kam.mff.cuni.cz}}
\address{Department of Applied Mathematics\\ Charles University\\ Prague, Czech Republic}
\author{Jan Hubi\v cka\thanksref{ERC}\thanksref{hubicka@iuuk.mff.cuni.cz}}
\address{Computer Science Institute of Charles University (IUUK)\\ Charles University\\ Prague, Czech Republic}
\author{Yangjing Long\thanksref{g2}\thanksref{yjlong@sjtu.edu.cn}}
\address{School of Mathematical Sciences\\ Shanghai Jiao Tong University\\Shanghai, China}
\thanks[g1]{Supported by M\v{S}MT \v{C}R grant LH12095 and GA\v{C}R grant P202/12/G061.}
\thanks[ERC]{Supported by grant ERC-CZ LL-1201 of the Czech Ministry of Education and CE-ITI P202/12/G061 of GA\v CR.}
\thanks[g2]{Supported by National Natural Science Foundation of China (No. 11271255)}
\thanks[fiala@kam.mff.cuni.cz]{Email: \href{mailto:fiala@kam.mff.cuni.cz} {\texttt{\normalshape{fiala@kam.mff.cuni.cz}}}}
\thanks[hubicka@iuuk.mff.cuni.cz]{Email: \href{mailto:hubicka@iuuk.mff.cuni.cz} {\texttt{\normalshape{hubicka@iuuk.mff.cuni.cz}}}}
\thanks[yjlong@sjtu.edu.cn]{Email: \href{mailto:yjlong@sjtu.edu.cn} {\texttt{\normalshape{yjlong@sjtu.edu.cn}}}}
\begin{abstract}
We characterise gaps in the full homomorphism order of graphs.
\end{abstract}
\end{frontmatter}
\maketitle
\begin{abstract}
We fully characterise pairs of finite graphs which form a gap in
the full homomorphism order.  This leads to a simple proof of the existence
of generalised duality pairs. We also discuss how such results can be carried to
relational structures with unary and binary relations.
\end{abstract}
\section{Introduction}

For given graphs $G=(V_G,E_G)$ and $H=(V_H,E_H)$ a {\em homomorphism} $f:G\to H$ is a
mapping $f:V_G\to V_H$ such that $\{u,v\}\in E_G$ implies $\{f(u),f(v)\}\in E_H$. (Thus it is an edge preserving mapping.)
The existence of a homomorphism $f:G\to H$ is traditionally denoted by $G\Hom H$.
This allows us to consider the existence of a homomorphism, $\Hom$, to be a
(binary) relation on the class of graphs. 
A homomorphism $f$ is {\em full} if $\{u,v\}\notin E_G$ implies $\{f(u),f(v)\}\notin
E_H$.  (Thus it is an edge and non-edge preserving mapping).  Similarly we will
denote by $G\FullHom H$ the existence of a full homomorphism $f:G\to H$.

As it is well known, the relations $\to$ and $\FullHom$ are reflexive (the identity is a
homomorphism) and transitive (a composition of two homomorphisms is still a
homomorphism). Thus the existence of a homomorphism as well as the existence of full homomorphisms induces a quasi-order on
the class of all finite graphs. We denote the quasi-order induced by
the existence of homomorphisms and the existence of full homomorphism on finite graphs by ($\Graphs,\leq)$ and ($\Graphs,\FullLeq)$ respectively.
(Thus when speaking of orders, we use $G\leq H$ in the same sense as $G\Hom H$ and $G\FullLeq H$ in the sense $G\FullHom H$.)

 These quasi-orders can be easily transformed
into  partial orders by choosing a particular representative for each
equivalence class. In the case of graph homomorphism such representative is up
to isomorphism unique vertex minimal element of each class, the {\em (graph) core}. In the case
of full homomorphisms we will speak of {\em F-core}.

The study of homomorphism order is a well established discipline and one
of main topics of nowadays classical monograph of Hell and
Ne\v{s}et\v{r}il~\cite{Hell2004}. The order $(\Graphs,\FullLeq)$ is a topic
of several publications \cite{Xie2006,Feder2008,Hell2013,Ball2010,Fiala} which 
are primarily concerned about the full homomorphism equivalent of the homomorphism
duality~\cite{Nesetril2000}.

In this work we further contribute to this line of research by characterising
{\em F-gaps} in $(\Graphs,\FullLeq)$. That is pairs of non-isomorphic F-cores $G\FullLeq H$ such
that  every F-core $H'$, $G\FullLeq H'\FullLeq H$, is isomorphic either to $G$ or $H$. We will show:
\begin{thm}
\label{thm:fullgap}
If $G$ and $H$ are F-cores and $(G,H)$ is an F-gap, then $G$ can be obtained from $H$
by removal of one vertex.
\end{thm}
First we show a known fact that F-cores correspond to point-determining graphs
which have been studied in 70's by Sumner~\cite{Sumner1973} (c.f. Feder and
Hell~\cite{Feder2008}). We also show that there is a full homomorphism between two F-cores if
and only if there is an embedding from one to another (see~\cite[Section 3]{Feder2008}). These two observations shed a lot of light into the nature of full homomorphism order and  makes the
characterisation of F-gaps look particularly innocent (clearly gaps in embedding
order are characterised by an equivalent of Theorem~\ref{thm:fullgap}).  The
arguments in this area are however surprisingly subtle. This becomes even more
apparent when one generalise the question to classes of graphs as done
by Hell and Hern\'andez-Cruz~\cite{Hell2013} where both results of
Sumner~\cite{Sumner1973} and Feder and Hell~\cite{Feder2008} are given for
digraphs by new arguments using what one could consider to be surprisingly
elaborate (and interesting) machinery needed to carry out the analysis.

We focus on minimising arguments about the actual structure of graphs and
use approach which generalises easily to digraphs and binary
relational structures in general (see Section~\ref{sec:relational}). In
Section~\ref{sec:pd} we outline the connection of point determining graphs and
F-cores. In Section~\ref{sec:gaps} we show proof of the main result.
In Section~\ref{sec:dualities} we show how the existence of gaps leads to a particularly
easy proof of the existence of generalised dualities (main results of~\cite{Xie2006,Feder2008,Hell2013,Ball2010}).

\section{F-cores are point-determining}
\label{sec:pd}
In a graph $G$, the {\em neighbourhood} of a vertex $v\in V_G$, denoted
by $N_G(v)$, is the set of all vertices $v'$ of $G$ such that $v$ is adjacent to $v'$ in $G$.  {\em Point-determining graphs} are
graphs in which no two vertices have the same neighbourhoods.  If we start with
any graph $G$, and gradually merge vertices with the same neighbourhoods, we
obtain a point-determining graph, denoted by $G_{\mathrm{pd}}$.

We write $G\FullHomeq H$ for any pairs of graphs such that $G\FullHom H$ and $H\FullHom G$.
It is easy to observe that $G_{\mathrm{pd}}$ is always an induced subgraph of $G$. Moreover, for every graph $G$ it holds that $G_{\mathrm{pd}}\FullHom G\FullHom G_{\mathrm{pd}}$ and thus $G\FullHomeq G_{\mathrm{pd}}$.  This motivates the following proposition:
\begin{prop}[\cite{Feder2008}]
\label{prop:F-core}
A finite graph $G$ is an F-core if and only if it is point-determining.
\end{prop}
\begin{proof}
Recall that $G$ is an F-core if it is minimal (in the number of vertices) within its
equivalence class of $\FullHomeq$. If $G$ is an F-core, $G_{\mathrm{pd}}$ can
not be smaller than $G$ and thus $G=G_{\mathrm{pd}}$.

It remains to show that every point-determining graph is an F-core.
Consider two point-determining graphs $G\FullHomeq H$ that are not isomorphic. There are 
full homomorphisms $f:G\FullHom H$ and $g:H\FullHom G$. Because injective full
homomorphisms are embeddings, it follows that either $f$ or $g$ is not
injective.  Without loss of generality, assume that $f$ is not injective. Consider $u,v\in V_G$, $u\neq v$, such
that $f(u)=f(v)$. Because full homomorphisms preserve both edges and non-edges,
the preimage of any edge is a complete bipartite graph.
If we apply this fact on edges incident with $f(u)$, 
we derive that $N_G(u)=N_G(v)$.
\end{proof}

\begin{prop}[\cite{Feder2008,hell2014connected}]
\label{prop:thincmp}
For F-cores $G$ and $H$ we have $G\FullHom H$ if and only if $G$ is an induced subgraph of $H$.
\end{prop}
\begin{proof}
Embedding is a special case of a full homomorphisms. In the opposite direction
consider a full homomorphism $f:G \FullHom H$. By the same
argument as in the proof of Proposition \ref{prop:F-core} we get that $f$ is injective,
as otherwise $G$ would not be point-determining.
\end{proof}

\section{Main result: characterisation of F-gaps}
\label{sec:gaps}
Given a graph $G$ and a vertex $v\in V_G$ we denote by $G\setminus v$ the graph created from $G$ by removing vertex $v$.
We say that vertex $v$ \emph{determines} a pair of vertices $u$ and $u'$ if  $N_{G\setminus v}(u)=N_{G\setminus v}(u')$.
This relation (pioneered in \cite{Feder2008} and used in~\cite{Feder2008,Hell2013,Xie2006})will play key role in our analysis.  
We make use of the following Lemma:

\begin{lem}
\label{lem:determining}
Given a graph $G$ and a subset $A$ of the set of vertices of $G$ denote by $L$ a graph on the vertices of $G$, where $u$ and $u'$ are adjacent if and only if there 
is $v\in A$  that determines $u$ and $u'$. Let $S$ be any spanning tree of $L$.
Denote by $B\subseteq A$ the set of vertices that determine some pair of vertices connected by an
edge of $L$ and by $C\subseteq B$ set of vertices that determine some pair of vertices connected by 
an edge of $S$. Then $B=C$.
\end{lem}
\begin{proof}
Because for every pair of vertices there is at most one vertex determining them clearly $C\subseteq B\subseteq A$.

Assume to the contrary that there is vertex $v\in B\setminus C$ and thus every
pair determined by $v$ is an edge of $L$ but not an edge of $S$.  Denote by $\{u,u'\}$  some such edge of $L$ determined by $v\in B$.
Adding this edge to $S$ closes a cycle. Denote by $u=v_1,
v_2,\ldots v_n=u'$ the vertices of $G$ such that every consecutive pair is an
edge of $S$.  Without loss of generality, we can assume that $v\in N_G(v_1)$ and $v\notin
N_G(v_n)$. Because $v\in N_G(v_i)$ implies $v\in N_G(v_{i+1})$ unless $v$
determines pair $\{v_i,v_{i+1}\}$ we also know that there is $1\leq i<n$
such that $v$ determines $v_i$ and $v_{i+1}$.  A contradiction with the fact that
${v_i, v_{i+1}}$ forms an edge of $S$.
\end{proof}
As a warmup we show the following theorem which also follows by \cite{Sumner1973} (also shown as Corollary 3.2 in \cite{Feder2008} for graphs and \cite{Hell2013} for digraphs):
\begin{thm}[\cite{Sumner1973,Feder2008,Hell2013}]
\label{thm:sub}
Every F-core $G$ with at least 2 vertices contains an $F$-core with $|V_G|-1$ vertices as an induced subgraph.
\end{thm}
\begin{proof}
Denote by $n$ number of vertices of $G$.
 If there is a vertex $v$ of $G$ such that the graph $G\setminus v$ is point-determining, 
it is the desired F-core.  Consider graph $S$ as in Lemma~\ref{lem:determining} where $A$ is the
vertex set of $G$. Because $S$ has at most $n-1$ edges and every edge of $S$ is determined
by at most one vertex, we know that there is vertex $v$ which does not determine any
pair of vertices and thus $G\setminus v$ is point-determining.
\end{proof}
In fact both \cite{Sumner1973,Hell2013} shows that every F-core $G$ with at least 2 vertices
contains vertices $v_1\neq v_2$ such that both $G\setminus v_1$ and $G\setminus v_2$ are F-cores.
This follows by our argument, too but needs bit more detailed analysis.  The main idea of the
following proof of Theorem~\ref{thm:fullgap} can also be adapted to show this.

\begin{proof}(of Theorem~\ref{thm:fullgap})
Assume to the contrary that there are F-cores $G$ and $H$ such that $(G,H)$ is
an F-gap, but $G$ differs from $H$ by more than one vertex.  By induction we construct two infinite sequences of vertices of $H$ denoted by 
$u_0,u_1,\ldots$ and $v_0,v_1,\ldots$ along with two infinite sequences of
induced subgraphs of $H$ denoted by $G_0,G_1,\ldots$ and $G'_0,G'_1,\ldots$ such that for every $i\geq 0$ it holds that:
\begin{enumerate}
\setlength\itemsep{0em}
\item $G_i$ and $G'_i$ are isomorphic to $G$,
\item $G_i$ does not contain $u_i$ and $v_i$,
\item $G'_i$ does not contain $u_i$ and $v_{i+1}$,
\item $u_i$ and $u_{i+1}$ is determined by $v_i$, and,
\item $v_i$ and $v_{i+1}$ is determined by $u_i$.
\end{enumerate}

Put $G_0=G$ and $A=V_H\setminus V_G$.
Consider the spanning tree $S$ given by Lemma~\ref{lem:determining}.
 Because no vertex of $A$ can be removed to obtain an induced
point-determining subgraph, it follows that every vertex must have a
corresponding edge in $S$. Consequently the number of edges of $S$ is at least
$|A|$. Because $G$ itself is point-determining, it follows that every edge of
$S$ must contain at least one vertex of $A$.  These two conditions yields to
the pair of vertices $v_0\in A=V_H\setminus V_G$ and $v_1\in V_G$ connected by an
edge in $S$ and consequently we have a vertex $u_0\in A$ which determines them. 
We have obtained $G_0, u_0, v_0, v_1$ with the desired properties.
This finishes the initial step of the induction.

\medskip

At the induction step assume we have constructed $G_i, u_i, v_i, v_{i+1}$.
We show the construction of $G'_i$ and $u_{i+1}$. We consider two cases.
If $v_{i+1}\notin V_{G_i}$ we put $G'_i=G_i$.  If $v_{i+1}\in V_{G_i}$ we let $G'_i$ to be
the graph induced by $H$ on $(V_{G_i}\setminus\{v_{i+1}\})\cup \{v_i\}$.
Because the neighbourhood of $v_i$ and $v_{i+1}$ differs only by a vertex $u_i\notin G_i$ which
determines them we know that $G'_i$ is isomorphic to $G_i$ (and thus also to  $G$) and moreover that $u_i$ is not a vertex of $G'_i$ (because $u_i\notin V_{G_i}$ can not determine itself and thus $u_i\neq v_i$). If $H$ was point-determining
after removal of $v_{i+1}$ we would obtain a contradiction similarly as before. We can thus assume that $v_{i+1}$
determines at least one pair of vertices.  Because neighbourhood $v_{i+1}$ and $v_i$ differs only by $u_i$
we know that one vertex of this pair is $u_i$. Denote by $u_{i+1}$ the second vertex. 

Given $G'_i, u_i, u_{i+1}, v_{i+1}$ we proceed similarly.
If $u_{i+1}\notin V_{G'_i}$ we put $G_{i+1}=G'_i$.  If $u_{i+1}\in V_{G'_i}$ we let $G_{i+1}$ to be
the graph induced by $H$ on $(V_{G'_i}\setminus\{u_{i+1}\})\cup \{u_i\}$.
Again $G_{i+1}$ is isomorphic to $G$ and does not contain $u_{i+1}$ nor $v_{i+1}$. Denote by $v_{i+2}$ a vertex determined by $u_{i+1}$ from
$v_{i+1}$ (which again must exist by our assumption) and we have obtained $G_{i+1}, u_{i+1}, v_{i+1}, v_{i+2}$ with the desired properties.
This finishes the inductive step of the construction.

\medskip

Because $H$ is finite, we know that both sequences $u_0,u_1,\ldots$  and
$v_0,v_1,\ldots$ contains repeated vertices.
Without loss of generality we can assume that repeated vertex with lowest index $j$ appears in the first sequence. We thus have $u_j=u_i$ for some $i< j$. By minimality of $j$ we can assume
that $v_i,v_{i+1},\ldots v_{j-1}$ are all unique. Assume that $v_i$ is
in the neighbourhood of $u_i$, then $v_i$ is not in the neighbourhood of $u_{i+1}$ (because it determines this pair) and consequently also $u_{i+1},u_{i+2},\ldots,u_{j}$.
A contradiction with $u_j=u_i$. If $v_i$ is not in the neighbourhood of $u_i$ we proceed analogously.
\end{proof}

\section{Generalised dualities always exist}
\label{sec:dualities}
To demonstrate the usefulness of Theorem~\ref{thm:fullgap} and Propositions~\ref{prop:F-core} and~\ref{prop:thincmp}
give a simple proof of the existence of generalised dualities in full homomorphism
order.
For two finite sets of graphs $\mathcal{F}$ and $\mathcal{D}$ we say that $(\mathcal{F},\mathcal{D})$ is a {\em generalised finite $F$-duality pair} (sometimes also {\em $\mathcal D$-obstruction}) if for any graph $G$ there exists $F\in \mathcal{F}$ such that $F\FullHom G$ if and only if $G\FullHom D$ for no $D\in \mathcal{D}$.

Existence of (generalised) dualities have several consequences. To mention one, it implies that the decision problem ``given graph $G$ is there $D\in \mathcal{D}$ and full homomorphism $G\to D$?'' is polynomial time solvable for every fixed finite family $\mathcal D$ of finite graphs.
In the graph homomorphism order the dualities (characterised in~\cite{Nesetril2000}) are rare. In the case of full homomorphisms they are however always guaranteed to exist.
\begin{thm}[\cite{Xie2006,Feder2008,Hell2013,Ball2010}]
For every finite set of graphs $\mathcal{D}$ there is a finite set of graphs
$\mathcal{F}$ such that $(\mathcal{F},\mathcal{D})$ is a generalised finite F-duality pair.
\end{thm}
\begin{proof}
Without loss of generality assume that $\mathcal D$ is a non-empty set of F-cores.
Consider set $\mathcal X$ of all F-cores $G$ such that there is $D\in \mathcal D$, $G\to D$.
Because, by Proposition~\ref{prop:thincmp},
 the number of vertices of every such $G$ is bounded from above by the number of vertices of $D$ and because $\mathcal D$
is finite, we know that $\mathcal X$ is finite.

Now denote by $\mathcal{F}$ the set of all F-cores $H$ such that $H\notin \mathcal X$ and there
is $G\in \mathcal X$ such that $(G,H)$ is a gap.  By Theorem~\ref{thm:fullgap} this set is finite.
We show that $(\mathcal{F},\mathcal{D})$ is a duality pair.

Consider an F-core $G$, either $G\in \mathcal X$ and thus there is $D\in \mathcal
D$, $G\to D$ or $G\notin \mathcal X$ and then consider a sequence of $F$-cores
$G_1,G_2,\ldots, G_{|G|}=G$ such that $G_1\in \mathcal X$ consists of single
vertex, $G_{i+1}$ is created from $G_i$ by adding a single vertex for every
$1\leq i<|G|$ (such sequence exists by Theorem~\ref{thm:sub}). Clearly there is $1\leq j<|G|$ such that $G_j\in \mathcal X$
and $G_{j+1}\notin \mathcal X$. Because $(G_j,G_{j+1})$ forms a gap, we know
that $G_{j+1}\in \mathcal F$.
\end{proof}
\begin{remark}
A stronger result is shown by Feder and Hell \cite[Theorem 3.1]{Feder2008} who shows
that if $\mathcal D$ consists of single graph $G$ with $k$ vertices, then $\mathcal F$ can
be chosen in a way so it contains graphs with at most $k+1$ vertices and there are at
most two graphs having precisely $k+1$ vertices.
While, by Theorem~\ref{thm:fullgap}, we can also give the same upper bound on number
of vertices of graphs in $\mathcal F$, it does not really follow that there
are at most two graphs needed. It appears that the full machinery of~\cite{Feder2008} is necessary
to prove this result.

In the opposite direction it does not seem to be possible to derive Theorem~\ref{thm:fullgap}
from this characterisation of dualities, because given pair of non-isomorphic F-cores $G\FullHom H$
and $\mathcal D$ a full homomorphism dual of $\{G\}$ it does not hold that for a graph $F\in \mathcal D$
such that $D\FullHom H$ there is also full homomorphism $G\FullHom H$.
\end{remark}
\section{Full homomorphisms of relational structures}
\label{sec:relational}
To the date, the full homomorphism order has been analysed in the context
of graphs and digraphs only.  Let us introduce generalised setting of relational
structures:

A language $L$ is a set of relational symbols $\rel{}{}\in L$, each associated with natural number $\arity{}$ called \emph{arity}.
A \emph{(relational) $L$-structure} $\str{A}$ is a pair $(A,(\rel{A}{};\rel{}{}\in L))$ where $\rel{A}{}\subseteq A^{\arity{}}$ (i.e. $\rel{A}{}$ is a $\arity{}$-ary relation on $A$). The set $A$ is called the \emph{vertex set} of $\str{A}$ and elements of $A$ are \emph{vertices}. The language is usually fixed and understood from the context.  If the set $A$ is finite we call \emph{$\str A$ finite structure}.  The class of all finite relational $L$-structures will be denoted by $\Rel(L)$.

A \emph{homomorphism} $f:\str{A}\to \str{B}=(B,(\rel{B}{};\rel{}{}\in L))$ is a mapping $f:A\to B$ satisfying for every $\rel{}{}\in L$ the implication $(x_1,x_2,\ldots, x_{\arity{}})\in \rel{A}{}\implies (f(x_1),f(x_2),\ldots,f(x_{\arity{}}))\in \rel{B}{}$.  A homomorphism is \emph{full} if the above implication is equivalence, i.e. if for every $\rel{}{}\in L$ we have $(x_1,x_2,\ldots, x_{\arity{}})\in \rel{A}{}\iff (f(x_1),f(x_2),\ldots,f(x_{\arity{}}))\in \rel{B}{}$. 

Given structure $\str{A}$ its vertex $v$ is contained in a {\em loop} if there exists $(v,v,\ldots,\allowbreak v)\in \rel{A}{}$ for some $\rel{}{}\in L$ of arity at least 2.
Given relation $\rel{A}{}$ we denote by $\overline{R}_\str{A}$ its complement, that is the set of all $\arity{}{}$-tuples $\vec{t}$ of vertices of $A$ that are not in $\rel{A}{}$.

When considering full homomorphism order in this context, the first problem is what should be considered to be the neighbourhood of a vertex. This can be described as follows:
Given $L$-structure $\str{A}$, relation $\rel{}{}\in L$ and vertex $v\in A$ such that $(v,v,\ldots, v)\notin \rel{A}{}$ the {\em $\rel{}{}$-neighbourhood} of $v$ in $\str{A}$, denoted by $N^{\rel{}{}}_\str{A}(v)$ is the set of all tuples $\vec{t}\setminus v$ created from
$\vec{t}\in \rel{A}{}$ containing $v$. Here by $\vec{t}\setminus v$ we denote tuple created from $\vec{t}$ by replacing all occurrences of vertex $v$ by a special symbol $\bullet$ which is not part of any vertex set. If $(v,v\ldots, v)\in \rel{A}{}$ then the $\rel{}{}$-neighbourhood $N^{\rel{}{}}_\str{A}(v)$ is the set of all tuples $\vec{t}\setminus v$ created from $\vec{t}\in \overline{R}_\str{A}\cup \{(v,v,\ldots, v)\}$. The {\em neighbourhood} of $v$ in $\str{A}$ is a function assigning every relational symbol its neighbourhood:
$N_\str{A}(v)(R)=N^{\rel{}{}}_\str{A}(v).$

We say that $L$-structure $\str{A}$ is {\em point-determining} if there are no
two vertices with same neighbourhood.  With these definitions direct analogies
of Proposition~\ref{prop:F-core} and \ref{prop:thincmp} for $\Rel(L)$ follows.

Analogies of Lemma~\ref{lem:determining}, Theorem~\ref{thm:sub} and Theorem~\ref{thm:fullgap} do not
follow for relational structures in general.  Consider, for example, a relational structure with three vertices $\{a,b,c\}$
and a single ternary relation $R$ containing one tuple $(a,b,c)$. Such structure is point-determining, but the only point-determining
substructures consist of single vertex.
There is however deeper problem with carrying Lemma~\ref{lem:determining} to relational structures: if a pair of
vertices $u,u'$ is determined by vertex $v$ their neighbourhood may differ by tuples containing additional vertices. Thus
the basic argument about cycles can not be directly applied here.  We consequently formulate results for relational language
consisting of unary and binary relations only (and, as a special case, to digraphs):
\begin{thm}
\label{thm:fullgap2}
Let $L$ be a language containing relational symbols of arity at most 2.  If $\str{A}$ and $\str{B}$ are (relational) F-cores and $(\str{A},\str{B})$ is an F-gap, then $\str{A}$ can be obtained from $\str{B}$ by removal of one vertex.
\end{thm}
The example above shows that the limit on arity of relational symbols is
actually necessary.  This may be seen as a surprise, because the results about
digraph homomorphism orders tend to generalise naturally to relational
structures and we thus close this paper by an open problem of characterising
gaps in full homomorphism order of relational structures in general.

\bibliographystyle{abbrv}
\bibliography{constrainedhomo}
\end{document}